\documentclass[letterpaper,11pt]{article}
\usepackage{amsmath, amsthm, amssymb}
\usepackage{bm}
\usepackage{graphicx}
\usepackage[colorlinks,linkcolor=blue,citecolor=blue,urlcolor=magenta,linktocpage,plainpages=false]{hyperref}\usepackage{xcolor}
\usepackage{fullpage}
\usepackage{xspace}
\usepackage{url}
\usepackage[noend]{algpseudocode}
\usepackage{libertine}
\usepackage{libertinust1math}
\usepackage[ruled,vlined]{algorithm2e}
\usepackage{caption}
\usepackage{setspace}
\usepackage{authblk}
\usepackage[utf8]{inputenc}
\usepackage[T1]{fontenc}

\newtheorem{lemma}{Lemma}
\newtheorem{thm}{Theorem}
\newtheorem{cor}{Corollary}
\newtheorem{prop}{Proposition}

\newtheorem{definition}{Definition}

\newcommand{\Z}[0]{\ensuremath \mathbb{Z}}

\newcommand{\R}[0]{\ensuremath \mathbb{R}}

\newcommand{\ip}[2]{#1  #2}

\newcommand{\ones}{\bm{1}}
\newcommand{\blue}[1]{\textcolor{blue}{#1}}
\newcommand{\red}[1]{{\color{red} #1}}

\newcommand{\rank}{\text{rank}}
\newcommand{\prank}{\text{prank}}
\newcommand{\modd}{~(\textup{mod}~2)}

\newcommand{\mom}[1]{{\left\vert\kern-0.25ex\left\vert\kern-0.25ex\left\vert #1 \right\vert\kern-0.25ex\right\vert\kern-0.25ex\right\vert}}

\makeatletter
\newcommand{\algmargin}{\the\ALG@thistlm}
\makeatother
\newlength{\whilewidth}
\algnewcommand{\parState}[1]{\State%
  \parbox[t]{\dimexpr\linewidth-\algmargin}{\strut #1\strut}}
\algtext*{Until}%

\newcounter{mynotes}
\newcommand{\conv}{\text{conv}}

\newcommand{\poly}{\text{poly}}

\title{On Polytopes with Linear Rank with respect to Generalizations of the Split Closure}
\author[1]{Sanjeeb Dash\thanks{sanjeebd@us.ibm.com}}
\author[2]{Yatharth Dubey\thanks{yatharthdubey7@gatech.edu}}
\affil[1]{IBM Thomas J. Watson Research Center}
\affil[2]{Georgia Institute of Technology}
\date{\today}

\begin{document}

\maketitle

\begin{abstract}
In this paper we study the rank of polytopes contained in the 0-1 cube with respect to $t$-branch split cuts and $t$-dimensional lattice cuts for a fixed positive integer $t$.
These inequalities are the same as split cuts when $t=1$ and generalize split cuts when $t > 1$.
For polytopes contained in the $n$-dimensional 0-1 cube, the work of Balas implies that the split rank can be at most $n$, and this bound is tight as Cornu\'ejols and Li gave an example with split rank $n$. All known examples with high split rank -- i.e., at least $cn$ for some positive constant $c < 1$ -- are defined by exponentially many (as a function of $n$) linear inequalities.
For any fixed integer $t > 0$, we give a family of polytopes contained in $[0,1]^n$ for sufficiently large $n$ such that each polytope has empty integer hull, is defined by $O(n)$ inequalities, and has rank $\Omega(n)$ with respect to $t$-dimensional lattice cuts.
Therefore the split rank of these polytopes is $\Omega(n)$.
It was shown earlier that there exist generalized branch-and-bound proofs, with logarithmic depth, of the nonexistence of integer points in these polytopes.
Therefore, our lower bound results on split rank show an exponential separation between the depth of branch-and-bound proofs and split rank.
\end{abstract}

\section{Introduction}

Split cuts, split closures, and split rank are important and widely studied concepts in cutting plane theory, the subarea of integer programming that deals with valid inequalities for integral points in polyhedra.
Let $P_I$ denote the integer hull of a polyhedron $P$, or the convex hull of integral points in $P$. A split cut for $P \subseteq \R^n$ is a linear inequality satisfied by the points $\{x \in P : \ip{\pi}{x} \leq \delta\}$ and $\{x \in P : \ip{\pi}{x} \geq \delta+1 \}$ for some $\pi \in \Z^n$ and $\delta \in \Z$ ($\pi x$ denotes the inner product between the vectors $\pi$ and $x$).
The set of points in $P$ satisfying all split cuts is called the {\em split closure} of $P$ and is denoted by $SC(P)$.
It was shown in \cite{split-closure} that the split closure of a rational polyhedron is a rational polyhedron.
Furthermore, it is known that the number of times the split closure operator needs to be repeated to obtain $P_I$ is finite.
Let $SC^k(P)$ be defined as follows:
\begin{equation}\label{eq:sck} SC^0(P) = P, \text{ and } SC^k(P) = SC(SC^{k-1}(P) \text{ for } k=1, 2, \ldots.
\end{equation}
The split rank of $P$ is the minimum $k \geq 0$ such that $SC^k(P) = P_I$. 

It follows from the work of Balas\cite{balas} that if a polytope $P$ is contained in the $n$-dimensional cube $[0,1]^n$, then the split rank of $P$ is at most $n$. 
Cornu\'ejols and Li \cite{cornujols2002rank} showed that the split rank of the $n$-dimensional ``cropped cube" contained in $[0,1]^n$ (see definition \ref{def-crop}) is exactly $n$.
The only other known polytopes contained in $\mathbb{R}^n$ with split rank $\Omega(n)$ are certain packing and covering polytopes contained in $[0,1]^n$, and the subtour elimination polytope associated with TSP problems \cite{dey2021lower, basu2021complexity}. These polytopes all have exponentially many (as a function of $n$) facets.

In this work, we show that a specific family of polytopes contained in $[0,1]^n$ and defined by $O(n)$ constraints has $\Omega(n)$ split rank.
We also consider generalizations of split cuts, namely $t$-branch split cuts \cite{li2008cook} and $t$-dimensional lattice cuts \cite{dash2020lattice} and their corresponding closure operators. We show that the above family of polytopes has linear $t$-dimensional lattice rank and $t$-branch split rank, for a fixed $t$, and obtain the result on split rank as a consequence.

These polytopes were studied in the context of other families of cutting planes in \cite{buresh2003rank}, namely Gomory-Chv\'atal (GC) cuts, and the matrix cuts based on the $N_0, N,$ and $N_+$ \cite{lovschrijver} operators.
We call them {\em Tseitin polytopes} and define them as follows.
Let $G = (V,E)$ be an undirected graph and for each vertex $u \in V$ let $N(u) = \{v \in V: uv \in E\}$.
\begin{definition}\label{ts-poly}
The \emph{Tseitin polytope} of a graph $G=(V,E)$, denoted by $P_{TS(G)}$, is contained in $[0,1]^{|E|}$ and described by the inequalities:
\begin{eqnarray*} &\sum_{v \in N(u) \setminus F} x_{uv} + \sum_{v \in F} (1 - x_{uv}) \geq 1 \quad &\forall~u \in V, \forall~F \subseteq N(u) \text{ with } |F| \text{ even},\\
&0 \leq x_{uv} \leq 1 &\forall~ uv \in E. 
\end{eqnarray*}
\end{definition}
The integer hull of a Tseitin polytope is known to be empty when $|V|$ is odd (we will explain this later).  If $G$ is $d$-regular, where $d$ is a constant, the polytope $P_{TS(G)}$ is described by $O(|V| \cdot 2^d) = O( |E| \cdot 2^d) = O(|E|)$ constraints, since $|E| = \frac{d}{2}|V|$. It follows from the work of \cite{buresh2003rank} that there is a constant $d$ such that for large enough odd values of $n$, there are $d$-regular graphs $G = (V,E)$ with $|V| = n$ such that $P_{TS(G)}$ is defined by $O(n)$ variables and inequalities, but has Chv\'atal rank (or $N$-rank or $N_+$-rank) $\Omega(n)$.  

Using the same family of polytopes, we prove that for any constant value of $t \geq 1$, there is some positive integer $n_0$ such that for all odd $n > n_0$, there exists a $20t$-regular graph with $n$ nodes and $10nt$ edges such that $P_{TS(G)}$ has $t$-lattice rank $\Omega(n)$.
Setting $t=1$, we get the desired lower bound on split rank.
This result on linear split rank also provides the first separation between depths of general branch-and-bound proof trees and split rank for binary problems. The former is always less than or equal to the latter. It was shown in \cite{beame2018stabbing} that, for all graphs $G$ with maximum degree $d$ and odd $|V|$, there is a general branch-and-bound proof (see Sections 1 and 2 of \cite{dey2021lower} for a formal definition) of the integer infeasibility of $P_{TS(G)}$ of depth at most $O(d + \log^2 |V|)$. In the situation where our split rank result holds ($n$ odd, $t=1$, $d = 20$), the depth of the branch-and-bound proof is at most $O(\log^2 n)$. 
Here we use the phrase "general branch-and-bound" to indicate that the branch-and-bound proof branches on general linear inequalities and not simply on variable bounds.   

A natural question is whether our lower bound on split rank can be tightened from $cn$ to $n - k$ where $0 < c < 1$ and $k$ is a nonnegative integral constant. We note that, for any constant $k$, there cannot be an integer infeasible polytope $P \subseteq [0,1]^n$ with $\poly(n)$ facets and split rank $n - k$. The fact that any such polytope would need at least $2^{n-k}$ facets follows directly from the proofs of Proposition 1 in \cite{eisenbrand1999bounds} and Proposition 4.3 in \cite{dash2001matrix}.
Therefore, to tighten our result, one would need to consider polytopes with nonempty integer hull.

The infeasible Tseitin polytopes also have linear rank with respect to a number of other operators/cutting planes: the {\em M-cuts} of Dunkel and Schulz \cite{ds}, 
the {\em two-halfspace closure} introduced by Basu and Jiang \cite{basu2020two}, and {\em semantic cutting planes} \cite{bpr}. It was shown in \cite{basu2020two} that the two-halfspace rank is greater than or equal to half the split rank. Therefore the linear rank result for the two-halfspace closure can be derived from our result on split rank.
The linear rank result for M-cuts follows from a result on the rank of semantic cutting planes in Fleming et. al. \cite{fleming2021power}. We discuss these results in detail in Section \ref{sec:bounds}.

In Section \ref{sec:prelim} we introduce the necessary notation and relevant concepts, and formally state the main results. In Section \ref{sec:bounds} we prove tight upper and lower rank bounds on two generalizations of split closure. Finally in Section \ref{sec:proof} we prove the main result of this paper. 

\section{Preliminaries}\label{sec:prelim}


Split cuts and the associated split closure were studied by Cook, Kannan, and Schrijver\cite{split-closure}.
We next give the definition of the $t$-branch split closure of a polyhedron $P$, introduced by 
Li and Richard \cite{li2008cook}. When $t=1$, their definition coincides with the definition of split closure in \cite{split-closure}.
\begin{definition}
Given a positive integer $t$ and a polyhedron $P \subseteq \R^n$, the \emph{$t$-branch split closure} of $P$ is denoted by $tSC(P)$ and defined as
\[tSC(P) = \mathop{\bigcap_{\pi^1,...,\pi^t \in \mathbb{Z}^{n}}}_{\pi^1_0, \ldots, \pi^t_0 \in \Z} \conv \left(\bigcap_{j = 1}^{t} \left(\{x \in P : \ip{\pi^j}{x} \leq \pi_0^j  \} \cup \{x \in P : \ip{\pi^j}{x} \geq \pi_0^j + 1 \}\right)\right)\]
\end{definition}
For convenience, we will let $SC(P)$ stand for $1SC(P)$.
We define $tSC^k(P)$ and $t$-branch split rank in a manner analogous to the definition of $SC^k(P)$ in (\ref{eq:sck}) and the split rank of a polyhedron $P$.
When the outer intersection above is taken over a single list of $t$ vectors , say $L = (\bar\pi^1,...,\bar\pi^t)$, and one list of $t$ integers $L_0 = (\bar\pi^1_0, \ldots, \bar\pi^t_0)$, we say that the resulting polyhedron is obtained by applying the {\em $t$-branch split disjunction} defined by $L$ and $L_0$ to $P$. 

Polyhedra with infinite $t$-branch split rank (with respect to an appropriately defined mixed-integer hull) were given by Li and Richard for $t=2$, and for all $t > 2$ by Dash and G\"unl\"uk \cite{dash2013t}. However, there is essentially no analysis of the $t$-branch split rank for pure integer problems, especially those defined by binary variables.

Cook, Kannan and Lov\'asz \cite{split-closure} gave an alternative definition of the split closure that is equivalent to the one above, and this was generalized by Dash, Dey, and G\"unl\"uk \cite{dash2011mixed} as follows.
\begin{definition}
Given a positive integer $t$ and a polyhedron $P \subseteq \R^n$, the \emph{$t$-dimensional lattice closure} of $P$ is denoted by $tLC(P)$ and defined as
$$tLC(P) = \bigcap_{\pi^1,...,\pi^t \in \mathbb{Z}^n} \conv \left(\bigcap_{j = 1}^{t} \{x \in P: \ip{\pi^j}{x} \in \mathbb{Z} \}\right).$$
\end{definition}

Let $t$ be a positive integer. Let $\pi^1,...,\pi^t \in \mathbb{Z}^{n}$ and let $\pi^1_0, \ldots, \pi^t_0 \in Z$.
Then for all $ 1\leq j \leq t$ we have
\[\{x \in \mathbb{R}^n : \ip{\pi^j}{x} \in \mathbb{Z} \} \subseteq \{x \in \mathbb{R}^n : \ip{\pi^j}{x} \leq \pi_0^j  \} \cup \{x \in \mathbb{R}^n : \ip{\pi^j}{x} \geq \pi_0^j + 1 \}.\]
Therefore, for any polyhedron $P \subseteq \R^n$ we have
\begin{equation} \label{tlcintsc}
tLC(P) \subseteq tSC(P).
\end{equation}
 Once again we define $tLC^k(P)$ and $t$-dimensional lattice rank in a manner similar to $SC^k(P)$ and split rank.
\begin{definition}\label{def:ts_cont}
The \emph{Tseitin contradiction} for a graph $G = (V, E)$ with odd $|V|$, denoted $TS(G)$, is the following: given a boolean variable $x_{uv}$ for each edge $uv \in E$, there exists no 0-1 assignment to all the variables satisfying
\begin{equation}\label{tsei-contr}  \sum_{v \in N(u)} x_{uv} \equiv 1\modd \quad \forall u \in V. \end{equation}
\end{definition}
We refer to an equation of the form (\ref{tsei-contr}) as an {\em odd parity} equation. Any solution of (\ref{tsei-contr}) would imply that there is a subgraph $G'$ of $G$, with vertex set equal to $V$ and edges corresponding to nonzero $x_{uv}$ values, such that every vertex in $G'$ has odd degree. Then the sum of vertex degrees would be odd for $G'$, which is not possible.
Note that $\bar x \in \{0,1\}^E$ is a solution of (\ref{tsei-contr}) if and only $\bar x \in P_{TS(G)}$. This is because the first set of constraints in Definition~\ref{ts-poly} prevent any solutions that have an even number of ones assigned to edges incident to a node $u$.

To state our main result, we need the following definition.
\begin{definition}
Let $G = (V,E)$ be a graph. For any two disjoint subsets of vertices $V_1, V_2$, let $e_G(V_1,V_2)$ denote the number of edges with one vertex in each of $V_1$ and $V_2$. Then, the \emph{edge expansion} of $G$ is 
$$\min_{S \subseteq V, |S| \leq \frac{|V|}{2}} \frac{e_{G}(S, V \setminus S)}{|S|}.$$
\end{definition}
When the graph $G$ is clear from the context, we will drop the subscript $G$ from $e_G$.
The following result can be viewed as a generalization of the similar result for Chv\'atal rank in \cite{buresh2003rank}; we use the same graphs and similar proof techniques.
\begin{thm}\label{thm:tdim_lb}
If $|V|$ is odd and $G = (V, E)$ is a graph with edge expansion $c > t + 1$, then $P_{TS(G)}$ has $t$-dimensional lattice rank at least $\frac{c-(t + 1)}{2t}|V|$. If $G$ is $d$-regular, then the $t$-dimensional lattice rank is at least $\frac{c-(t + 1)}{dt}|E|$. 
\end{thm}
The bound in Theorem \ref{thm:tdim_lb} also holds for the $t$-branch split rank by Equation \ref{tlcintsc}. When $t=1$, this gives a lower bound on split rank for Tseitin polytopes. As GC cuts are a subclass of split cuts, the Chv\'atal rank is at least as much as the split rank. 

\begin{cor}\label{cor:split_lb}
If $|V|$ is odd and $G = (V, E)$ is a graph with edge expansion $c > 2$, then $P_{TS(G)}$ has split rank at least $\frac{c - 2}{2} |V|$. If $G$ is $d$-regular, then the split rank is at least $\frac{c - 2}{d} |E|$. 
\end{cor}

In order to get a linear lower bound on $t$-dimensional lattice rank, it suffices to have $c$ and $d$ as constants independent of $|V|$. The following result from \cite{ellis2011expansion} yields this property.
\begin{prop}  \label{prop:ellis}
For any $d$, there exists an $n_0$ such that almost all $d$-regular graphs on at least $n_0$ vertices have edge expansion at least $0.18d$. 
\end{prop}
The above result implies that for any constant $t$ and sufficiently large odd $n$, almost all $20t$-regular graphs with $n$ nodes have edge expansion at least $3.6t$. Therefore, their corresponding Tseitin polytopes, defined on $|E| = 10tn$ variables, have $O(|E|)$ facets, empty integer hull, and $t$-dimensional lattice rank $\Omega(|E|)$. 
\begin{cor}\label{cor:main-cor}
For any $t$, there exists an $n_0$ such that for all $n > n_0$, there exist polytopes $P \subseteq [0,1]^{n}$ with $O(n)$ facets such that the split rank of $P$ is $\Omega(n)$. 
\end{cor}
Note that the discussion before Corollary \ref{cor:main-cor} refers only to polytopes contained in $[0,1]^{10tn}$, and not $[0,1]^n$ where $n$ is not a multiple of $10t$. However, it is not hard to obtain Corollary~\ref{cor:main-cor} for each $n$. Consider any $n$ that is not a multiple of $10t$ and let $n'' < n$ be the maximum integer such that $n''$ is a multiple of $10t$ (so $n'' > n - 10t$, and therefore $n'' = \Theta(n)$). Then, by the discussion above, we know that there is a polytope $P \subseteq [0,1]^{n''} \subseteq [0,1]^{n}$ with $O(n'') = O(n)$ facets and split rank $\Omega(n'') = \Omega(n)$.


\section{Bounds on rank}\label{sec:bounds}

We start off by giving some easy upper bounds on the t-branch split rank and $t$-dimensional lattice rank of a polytope $P \subseteq [0,1]^n$.
\begin{thm}
Let $P \subseteq [0,1]^n$ be a polytope and let $0 < t \leq n$ be a positive integer. Then, the $t$-branch split rank  of $P$ is at most $\lceil \frac{n}{t} \rceil$, and so is the $t$-dimensional lattice rank. 
\end{thm}
\begin{proof}
Let $P$ and $t$ satisfy the conditions of the Lemma. Partition the ordered list $(1 \dots, n)$ into $\lceil \frac{n}{t} \rceil$ lists so that the $k$th list $G_k = ( kt + 1,..., \min((k+1)t, n) \,)$ for $k = 0,...,\lceil \frac{n}{t} \rceil - 1$. Let $e^j$ denote the $j$-th unit vector. For any polytope $Q \subseteq [0,1]^n$, let
$$Q^k = \conv \left(\bigcap_{j \in G_k} \left(\{x \in Q : \ip{e^j}{x} \leq 0  \} \cup \{x \in Q: \ip{e^j}{x} \geq 1 \}\right)\right).$$
One can see that $Q^k$ is obtained by applying the $t$-branch split disjunction defined by $L_k = (e^j : j \in G_k)$ and $L'_k = (0, \ldots, 0)$ to $Q$.
By definition we have $tSC(P) \subseteq P^k$. Also, observe that $P^k \subseteq \conv(\{x \in P : x_j \in \{0,1\} \; \forall \; j \in G_k\})$. From this we can conclude that $((P^1)...)^{\lceil \frac{n}{t} \rceil} \subseteq \conv( P \cap \{0,1\}^n)$ and the first part of the result follows. As $tLC(P) \subseteq tSC(P)$, the second part of the result follows.
\end{proof}

Given $x \in \{0,\frac{1}{2}, 1\}^n$, let $E(x)$ be the set of indices where $x$ is fractional i.e., $E(x) = \{j \in \{1, \ldots, n\} : x_j = \frac{1}{2}\}$. For a set of indices $J \subseteq E(x)$, let $x^{(J,0)}$ and $x^{(J,1)}$ be defined as follows:
\[ x^{(J,0)}_i = \left\{ \begin{array}{ll} 0 & \text{ if } i \in J \\ x_i & \text{ otherwise} \end{array}\right. ~\text{ and }~ x^{(J,1)}_i = \left\{ \begin{array}{ll} 1 & \text{ if } i \in J \\ x_i & \text{ otherwise.} \end{array}\right.\]
If $j$ is a number from $\{1, \ldots, n\}$, we let $x^{(j,0)}$ stand for $x^{(\{j\}, 0)}$.
Let $\mathcal C$ be a family of cutting planes. We let $\mathcal{C}(P)$ stand for the set of points in $P$ that satisfy all cuts in $\mathcal{C}$, and we define $\mathcal{C}^k(P)$ in the same way we defined $SC^k(P)$. 
\begin{definition}
We say that a family of cutting planes $\mathcal{C}$ has the {\em t-rounding property} if for all $x \in \{0,\frac{1}{2}, 1\}^n$ with $|E(x)| \geq t$ the following holds:
if for all $J \subseteq E(x)$ with $|J| \leq t$ the points $x^{(J,0)}$ and $x^{(J,1)}$ are contained in a polytope $P \subseteq [0,1]^n$, then $x \in \mathcal{C}(P)$.
\end{definition}
The above definition says that whenever the points obtained by rounding $J$ fractional components of a half-integral point $x$ up to 1 or down to 0 are contained in $P$, then $x$ is contained in the closure of $P$ with respect to the family of cuts $\mathcal{C}$.  

The 1-rounding property is known to be true for GC cuts (Chv\'atal, Cook and Hartmann \cite{cch}) for $N_0, N,$ and $N_+$ cuts (Cook and Dash \cite{cd} and Goemans and Tun\c{c}el \cite{gt}), and for split cuts (Cornu\'ejols and Li \cite[Lemma 5]{cornujols2002rank}).
This property has been used to show that the closure operators associated with the above cut families all have rank at least $n$ for the following family of polytopes.
\begin{definition}\label{def-crop}
The $n$-dimensional \emph{cropped cube} is defined by the constraints
\begin{eqnarray*} && \sum_{i \in J}x_i + \sum_{i \not \in J}(1 - x_i) \geq \frac{1}{2} \quad \forall J \subseteq \{1,...,n\},\\
&&0 \leq x_i \leq 1 \quad \forall i \in \{1,...,n\}. 
\end{eqnarray*}
\end{definition}

We will next show that $t$-dimensional lattice cuts have the $t$-rounding property. We  will need the following basic property of the vector space  $\mathbb{F}_2^n$ where $\mathbb{F}_2$ is the finite field of integers modulo 2.
\begin{lemma}\label{lem:lin_alg}
Let $1 \leq t \leq n$, and let $\ones$ be the all-ones vector with $n$ components. Let $A \in \{0,1\}^{t \times n}$ be a matrix and let $b \in \{0,1\}^t$ be a vector such that $A\ones \equiv b\modd$. Then there exists a set $J \subseteq \{1, \ldots, n\}$  with $|J| \leq t$ such that $A_J \ones \equiv b\modd$, where $A_J$ is the submatrix of $A$ induced by column indices in $J$.
\end{lemma}

\begin{thm}\label{thm:t_latt_cl}
Let $t$ be a positive integer. The family of $t$-dimensional lattice cuts has the $t$-rounding property.
\end{thm}
\begin{proof}
Let $1 \leq t \leq n$ and let $x \in \{0, \frac{1}{2}, 1\}^n$ with $|E(x)| \geq t$. Let $P \subseteq [0,1]^n$ be a polytope and assume that $x^{(J, 0)}$ and $x^{(J,1)}$ are both contained in $P$ for all $J \subseteq E(x)$. We will prove that $x \in tLC(P)$.

Consider an arbitrary collection of $t$ vectors $\pi^1,...,\pi^t \in \mathbb{Z}^n$.
Let the ordered list of components of $E(x)$ be $(k_1, k_2, \ldots, k_{|E(x)|})$.
Let $\Pi$ be the $t\times |E(x)|$ matrix with 0-1 components defined by 
$$\Pi_{ij} = \begin{cases}
0 & \text{ if } \pi^i_{k_j} \text{ is even}, \\
1 & \text{ if } \pi^i_{k_j} \text{ is odd}.
\end{cases}$$
In other words $\Pi_{ij} \equiv 2\pi^i_{k_j}\modd$.
Let $b \equiv \Pi \cdot \ones\modd$ be a 0-1 vector. Notice that $b_i$ indicates the fractionality of $\ip{\pi^i}{x}$, i.e., $b_i = 0$ if $\ip{\pi^i}{x} \in \Z$ and $b_i = 1$ otherwise. Equivalently, $b_i \equiv \pi^i(2x) \modd$. Now, let $J$ be the set of at most $t$ columns of $\Pi$ shown to exist in Lemma \ref{lem:lin_alg}, so that $b \equiv \Pi_J \ones \modd$. 

We will now show that $\ip{\pi^1}{x^{(J,0)}},...,\ip{\pi^t}{x^{(J,0)}} \in \Z$. Consider some $i \in \{1, \ldots, t\}$ and notice that
\[\ip{\pi^i}{x^{(J,0)}} = \ip{\pi^i}{x} - \frac{1}{2}\sum_{j \in J}\pi^i_j.\]
Assume $\ip{\pi^i}{x} \in \Z$ and $b_i = 0$. Then by Lemma \ref{lem:lin_alg}, we know $\sum_{j \in J}\pi^i_j \equiv 0 \modd$, and therefore $\frac{1}{2}\sum_{j \in J}\pi^i_j \in \mathbb{Z}$. So $\ip{\pi^i}{x^{(J,0)}} \in \mathbb{Z}$ as desired. Next assume $\ip{\pi^i}{x}$ is fractional or $b_i = 1$. By Lemma \ref{lem:lin_alg}, we know $\sum_{j \in J}\pi^i_j \equiv 1 \modd$, and therefore  $\ip{\pi^i}{x^{(J,0)}} \in \mathbb{Z}$ as desired.
One can similarly argue that 
$$\ip{\pi^i}{x^{(J,1)}} = \ip{\pi^i}{x} + \frac{1}{2}\sum_{j \in J}\pi^i_j \in \Z.$$

We have shown that $x^{(J,0)}, x^{(J,1)} \in \{x : \ip{\pi^1}{x},...,\ip{\pi^t}{x} \in \mathbb{Z} \}$. As we assumed that $x^{(J,0)}, x^{(J,1)} \in P$, it follows that $x \in \conv(x^{(J,0)}, x^{(J,1)}) \subseteq \conv(\{x \in P : \ip{\pi^1}{x},...,\ip{\pi^t}{x} \in \mathbb{Z} \})$. 

As the choice of $\pi^1, \ldots, \pi^t$ was arbitrary, we can conclude that
\[ x \in \bigcap_{\pi^1,...,\pi^t \in \mathbb{Z}^n} \conv(\{x \in P : \ip{\pi^1}{x},...,\ip{\pi^t}{x} \in \mathbb{Z} \}) = tLC(P).
\]
\end{proof}

\begin{thm}
Let $P \subseteq [0,1]^n$ be the cropped cube. Then the $t$-dimensional lattice rank of $P$ is at least $\lceil \frac{n}{t} \rceil$.
\end{thm}
\begin{proof}
Let $x \in \{0,\frac{1}{2}, 1\}^n$. Observe if $|E(x)| \geq 1$ (i.e. $x$ has at least one fractional component), $x \in P$. For the sake of induction, assume that for some nonnegative integer $k$ we have $\{x \in \{0,\frac{1}{2}, 1\}^n : |E(x)| \geq kt + 1\} \subseteq tLC^k(P)$. Consider any $x$ such that $|E(x)| \geq (k+1)t + 1$. By the inductive hypothesis, for any subset $J \subseteq E(x)$ with size at most $t$, it holds that $x^{(J,0)}, x^{(J,1)} \in tLC^k(P)$. Applying Theorem \ref{thm:t_latt_cl}, we see that $x \in tLC^{k+1}(P)$. This implies $\frac{1}{2}\ones \in tLC^{\lceil \frac{n}{t} \rceil - 1}(P)$, and so the result follows.
\end{proof}
As $t$-branch split cuts are dominated by $t$-dimensional lattice cuts (see (\ref{tlcintsc})), we get the next result as an immediate consequence of the previous theorem.
\begin{cor}
Let $P \subseteq [0,1]^n$ be the cropped cube. Then the $t$-branch split rank of $P$ is at least $\lceil \frac{n}{t} \rceil$.
\end{cor}


We next define {\em multi-dimensional knapsack cuts} and the associated closure. We define $\mathcal{V}(P) = \{ (\pi, \pi_0) \in \mathbb{Z}^n \times \mathbb{Z} :  \ip{\pi}{x} \leq \pi_0 \; \forall x \in P \}$ to be the set of valid inequalities for $P$.
\begin{definition}
Given a polytope $P \subseteq [0,1]^n$ and a positive integer $t$, the $t$-dimensional knapsack closure of $P$ is the set
\[ tKC(P) = \bigcap_{(\pi^1,\pi^1_0), \ldots, (\pi^t,\pi^t_0) \in \mathcal{V}(P)} \conv{\left(\{x \in \{0,1\}^n : \pi^1 x \leq \pi^1_0, \ldots, \pi^t x \leq \pi^t_0 \}\right)}. \]
\end{definition}
We refer to valid inequalities for $tKC(P)$ as $t$-dimensional knapsack cuts.
The knapsack closure studied in \cite{fischetti2010knapsack} is the same object as $1KC(P)$.
Similarly {\em semantic cutting planes} and the associated closure operator are equivalent to $2KC(P)$.
Dunkel and Schulz \cite{ds} gave a strengthening of GC cuts for polytopes contained in the 0-1 cube and called these M-cuts; these are contained in the family of 1-dimensional knapsack cuts. Therefore the associated closure (M-closure) contains $1KC(P)$.
We next show that 1-dimensional knapsack cuts have the 1-rounding property using an easy counting argument.

\begin{lemma}\label{lem:zero_one_pts}
Let $Y \subseteq \{0,1\}^n$ and $|Y| \geq 2^{n-1} + 1$. Then there exist two points $y^1, y^2 \in Y$ such that $\frac{1}{2}(y^1 + y^2) = \frac{1}{2}\ones$.
\end{lemma}
\begin{proof}
Note that $f:\{0,1\}^n \rightarrow \{0,1\}^n$ defined by $f(y) = \ones - y$ is a bijective function. Consider a set $Y \subseteq \{0,1\}^n$ of size at least $2^{n-1} + 1$. Then $|f(Y)| \geq 2^{n-1} + 1$ and therefore $Y \cap f(Y) \not = \emptyset$. In other words, $Y$ contains $y^1,y^2$ such that $y^2 = \ones - y^1$. The desired result follows.
\end{proof}
 
\begin{thm}\label{thm:zero_one_cl}
Let $P \subseteq [0,1]^n$ be a polytope and let $\bar{x} \in \{0, \frac{1}{2}, 1\}^n$ with $|E(x)| \geq 1$. If $\bar{x}^{(i,0)}, \bar{x}^{(i, 1)} \in P$ for every $i \in E(\bar{x})$, then $\bar x \in 1KC(P)$.
\end{thm}
\begin{proof}
Let $d = |E(\bar{x})|$. Note that $\bar x = \frac{1}{2}(\bar{x}^{(i,0)} + \bar{x}^{(i, 1)})$ for some $i \in E(\bar x)$ (which is nonempty) and therefore $\bar x \in P$.
Let $F$ denote the face of the $n$-dimensional unit cube given by $\{x \in [0,1]^n : x_j = \bar{x}_j \; \forall j \not \in E(\bar{x}) \}$.
Consider a $(\pi, \pi_0) \in \mathcal{V}$. Let $U =\{x \in F \cap \{0,1\}^n : \ip{\pi}{x} \leq \pi_0 \}$. 

Assume that $|U| \leq 2^{d - 1}$, and let $Y =\{x \in F \cap \{0,1\}^n : \ip{\pi}{x} \geq \pi_0 + 1\}$. Then we have $Y \cap U = \emptyset$ and $| Y \cup U| = 2^d$ and therefore $|Y| \geq 2^{d-1}$.
Define the subset $Y_{j,0} = \{y \in Y : y_j = 0\}$ for some $j \in E(\bar x)$. Suppose $|Y_{j,0}| \geq 2^{d-2} + 1$ and let $y^*$ be defined as follows: $y^*_j = 0$ and $y^*_i = \frac{1}{2}$ for all $i \in E(\bar x) \setminus \{j\}$. By Lemma \ref{lem:zero_one_pts} applied to the $d-1$-dimensional unit cube formed by the face $\{y \in F : y_j = 0\}$, we have $y^1, y^2 \in Y_{j,0}$ such that $y^* = (y^1+y^2)/2 \Rightarrow \ip{\pi}{y^*} \geq \pi_0 + 1$. However, we know this cannot be the case, since $y^* = x^{(j,0)} \in P$ by the main assumption of the Theorem. Therefore, $|Y_{j,0}| \leq 2^{d-2}$ and by a symmetrical argument $|Y_{j,1}| \leq 2^{d-2}$. Since $2^{d-1} \leq |Y| = |Y_{j, 0}| + |Y_{j, 1}| \leq 2^{d-1} $, it follows that $|Y_{j, 0}| = |Y_{j, 1}| = 2^{d-2}$ and therefore the $j$-th coordinate of all $y \in Y$ averages to $\frac{1}{2}$. Moreover, the same argument holds for all $j \in E(x)$, and therefore the average of all $y \in Y$ is $\bar{x}$. This implies that $\ip{\pi}{\bar{x}} \geq \pi_0 + 1$, but this contradicts the fact that $\pi x \leq \pi_0$ is a valid inequality for $P$ and $\bar{x} \in P$.

Therefore we can assume that $|U| \geq 2^{d-1} + 1$. Then, by Lemma \ref{lem:zero_one_pts} applied to the hypercube formed by $F$, there exist $y^1, y^2 \in U$ such that $\frac{1}{2}(y^1 + y^2) = \bar{x}$, and therefore $\bar{x} \in \conv(\{x \in \{0,1\}^n : \ip{\pi}{x} \leq \pi_0 \})$.
As the choice of $(\pi, \pi_0)$ was arbitrary, the result follows.
\end{proof}

A consequence of the previous result is that the rank of the cropped cube with respect to $1KC(P)$ is exactly $n$.
We do not know if a $k$-rounding property holds for $tKC(P)$ for some constant $k$ when $t > 1$. 
Even without such a rounding property it is not hard to show that the rank of the $n$-dimensional cropped cube with respect to $tKC(P)$ is high.
The cropped cube is defined by $2^n$ inequalities, all of which are necessary for the emptiness of the integer hull.
Standard arguments in \cite{cch} can be used to show that any cutting-plane proof of infeasibilty using $t$-dimensional knapsack cuts must have length at least $2^n/(nt)$ and at most $(nt)^{r+1}$ inequalities, if $r$ is the the rank of the final inequality, and therefore $r = \Omega(n/\log nt)$. In other words, the fact that the cropped cube is defined by exponentially many inequalities and a cutting plane from a family $\mathcal{C}$ is derived from a small number of previous inequalities can be used to show that the rank of the cropped cube is high with respect to $\mathcal{C}$.

We will next discuss more interesting lower bound results for polytopes with a polynomial number of inequalities, namely the Tseitin polytopes defined on bounded-degree graphs. The above simple counting arguments are not sufficient to prove linear lower bounds on rank for these polytopes.

\section{Linear lower bound on rank}\label{sec:proof}
The 1-rounding property of GC, $N_0, N,$ and  $N_+$ cuts was used in \cite{buresh2003rank} to show that the rank of certain Tseitin polytopes is linear (as a function of dimension) for all these operators. We observe that their proof technique can be used along with the 1-rounding property of split cuts and 1-dimensional knapsack cuts to obtain a similar linear lower bound on rank for these cutting plane families. The latter result also follows from a linear lower bound on rank with respect to 2-dimensional knapsack cuts proved recently in \cite{fleming2021power}. However, our observation on split cut rank is new. Instead of showing this, we instead prove a more general result showing a similar lower bound for $t$-dimensional lattice cuts. To do this, we will modify the proof in \cite{buresh2003rank}. We will refer to the earlier paper for parts of the proof that we do not need to change.

We now explain an idea to show lower bounds on rank for families $\mathcal{C}$ that have the $t$-rounding property for some fixed $t$.
Consider a point $x \in P \setminus P_I$ and define $\prank(x)$ as
\[ \prank(x) = \max\{k : x \in \mathcal{C}^k(P)\}.\]
Clearly $k \geq 0$ and the rank of $P$ with respect to $\mathcal{C}$ is at least $prank(x)+1$. 
We next define a \emph{$\{0,1/2\}$-certificate tree}, a type of certificate for the rank of a polytope that is structured as follows. Each node in the tree is labeled by a point contained in $P \cap \{0, \frac{1}{2}, 1\}^n$.  Each non-leaf node corresponds to a point $x$ and it either has
\begin{enumerate}

    \item $|E(x)| \geq t$, and $2\binom{|E(x)|}{t}$ {\em red children} corresponding to the points $x^{(J,0)}, x^{(J,1)}$ for all  $J \subseteq E(x), |J| \leq t$ or
    \item $|Y|$ \emph{blue children} corresponding to points $y \in Y \subseteq \{0,\frac{1}{2}, 1\}^n$ where $x \in \conv(Y)\setminus Y$.
\end{enumerate}
The following bounds on $\prank(x)$ can be obtained in case it has red or blue children, respectively:
\begin{align}
\prank(x) &\geq 1 + \min_{J \subseteq E(x), |J| \leq t} \min \{\prank(x^{(J, 0)}),  \prank(x^{(J, 1)}) \},\label{eq:induct}\\
\prank(x) &\geq \min_{y \in Y} \prank(y) \label{eq:conv_pts}.
\end{align}
Equation (\ref{eq:induct}) follows from the $t$-rounding property of $\mathcal{C}$.
If there exists a path from the root to a leaf that passes through a red child node at least $k$ times, then equations (\ref{eq:induct}) and (\ref{eq:conv_pts}) together imply that $\prank(x^*) \geq k$, where the root is labeled by the point $x^*$. The tree can be viewed as a certificate of this fact. See Figure \ref{fig:certificate_tree} for an example with $x^* = \frac{1}{2}\ones$.

\begin{figure}[ht]
    \centering
    \includegraphics[width=\textwidth]{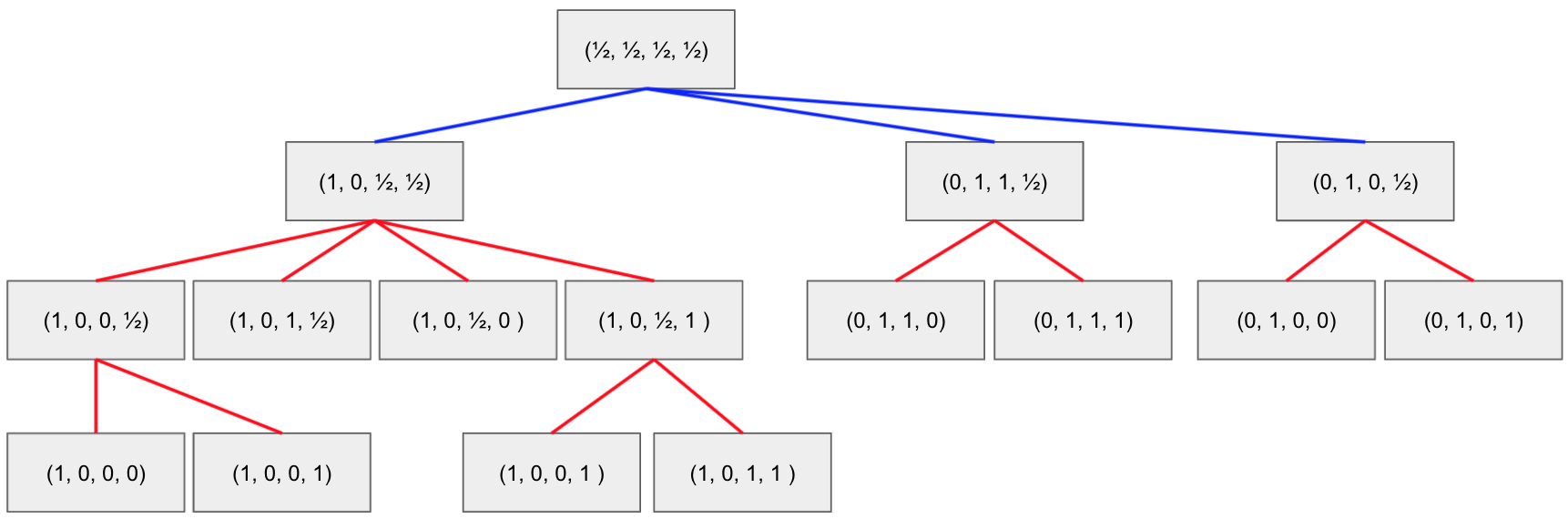}
    \caption{Let $P$ be the convex hull of the points $(1,0,0,0), (1,0,0,1), (1,0,1,1), (0,1,1,0), (0,1,1,1), (0,1,0,0),$ $(0,1,0,1), (1,0,1,\frac{1}{2}), (1,0,\frac{1}{2},1)$. Then notice the above tree certifies $\prank(\frac{1}{2}\ones) \geq 2$. }
    \label{fig:certificate_tree}
\end{figure}

If one considers the certificate tree implicitly defined by the proof that the $n$-dimensional cropped cube has rank $\lceil n/t \rceil$ for cutting plane families that have the $t$-rounding property, one can see that it is uniform in the sense that every path from the root to a leaf has the same number of red nodes, and for nodes at a given depth, all children are of the same type.
In the proof that Tseitin polytopes have linear rank for GC cuts in \cite{buresh2003rank}, the resulting certificate tree are not uniform: nodes at a certain depth may not all have the same color children.

We introduce some notation. Given a graph $G$, let $V(G)$ and $E(G)$ denote its vertex set and edge set, respectively. If $\bar x \in \{0,\frac{1}{2}, 1\}^{E(G)}$ is a half-integral vector, we index its components by the edges in $E(G)$. Then we say that an edge $uv$ is a 0-edge, 1-edge or $1/2$-edge if $\bar x_{uv} = 0, 1,$ or $1/2$, respectively. We let $E(\bar x) = \{uv \in E(G) : \bar x_{uv} = 1/2\}$. We define $H(\bar x)$ to be the subgraph of $G$ where $E(H(\bar x)) = E(\bar x)$ and $V(H(\bar x)) = \{v \in V(G) : \exists uv \in E(\bar x)\}$. In other words, a vertex of $G$ is a vertex of $H(\bar x)$ if there is a 1/2-edge incident to it. Therefore the remaining vertices of $G$ only have 0-edges or 1-edges incident to it. We shorten $V(H(\bar x))$ to $V(\bar x)$ for convenience. We let $e_x(U,V)$ stand for $e_{H(x)}(U,V)$, where $e()$ is the edge expansion function.
Let $P$ be the Tseitin polytope of $G$. We can say that $\bar x \in P$ if
\begin{eqnarray}
   & \mbox{each vertex in } V(\bar x) \mbox{ has at least two  1/2-edges incident to it,} \label{one-x-inP}\\
& \mbox{ each vertex of } G \mbox{ that is not in } V(\bar x) \mbox{ has an odd number of 1-edges incident to it.} \label{two-x-inP} 
\end{eqnarray}


\begin{lemma}\label{lem:avg_of_sols} \cite{buresh2003rank}
Consider the system $Ax \equiv b \modd$ of modular equations over $n$ variables where $A$ is a binary matrix and $b$ is a binary vector. Assume that for every $j \in \{1, \ldots, n\}$, there are solutions $y$ and $y'$ such that  $y_j = 0$ and $y_j' = 1$. Then the average of all $0-1$ solutions is $\frac{1}{2}\ones$. 
\end{lemma}
The next result is an adapted version Lemma 4.6 of \cite{buresh2003rank}, and its proof follows as well.
\begin{lemma}\label{lem:exist_sol}
Let $G = (V,E)$ be a graph and let $U \subseteq V$ be a subset of vertices with $|U| \leq |V|/2$.
If $e(X, V \setminus X) > 0$ for all $X \subseteq U$, then, for each $b \in \{0,1\}^U$, the following system of modular equations has a 0-1 solution:
\[ \sum_{v \in N(u)} x_{uv} \equiv b_u \modd\quad \forall u \in U. \]
\end{lemma}

We next prove our main result which yields a lower bound on rank for any family of cutting planes with the $t$-rounding property. As $t$-dimensional lattice cuts have the $t$-rounding property, Theorem ~\ref{thm:tdim_lb} follows.
\begin{lemma} Let $P$ be the Tseitin polytope of a graph $G$ with edge expansion $c > t + 1$.  Let $\mathcal{C}$ be a family of cutting planes with the $t$-rounding property. Then $P$ has rank at least $\frac{(c-(t + 1))}{2t}|V|$ with respect to $\mathcal{C}$. 
\end{lemma}
\begin{proof}
To prove the result, we will construct a $\{0, 1/2\}$-certificate tree $\mathcal{T}$ with the root labeled by $\frac{1}{2}\ones$ such that a path from the root to any leaf visits a red child at least $\frac{c-(t + 1)}{2t}|V|$ times.
As an intermediate step, we will construct an auxiliary tree $\mathcal{T}^A$ which will have the property that deleting its leaves results in the desired certificate tree. 
Let the root of $\mathcal{T}^A$ be labeled by $x^* = \frac{1}{2}\ones$ where $\ones$ has $E(G)$ components.
As $x^*$ satisfies the conditions (\ref{one-x-inP}) and (\ref{two-x-inP}) (the set $V(G) \setminus V(x^*)$ is empty), we observe that $x^* \in P$.
Assume that every node in $\mathcal{T}^A$, labeled by a vector $x$, has an associated nonnegative integer $\ell_x$.
Let $l_{x^*} = |V|/2$.

We next inductively define how to construct $\mathcal{T}^A$.
Consider an already constructed node in $\mathcal{T}^A$ labeled by the point $x \in P \cap \{0, \frac{1}{2}, 1\}^{E(G)}$. Let $\ell_x > 0$ and $|E(x)| \geq t$.
By definition, all inequalities defining $P$ are satisfied by $x$ and condition (\ref{two-x-inP}) is satisfied.

Assume that for all $U \subseteq V(x)$ of size at most $\ell_x$ we have $e_x(U, V(x) \setminus U) > (t + 1)|U|$. 
In this case we give the node $x$ red children. Let $y$ be the label of a red child of $x$. By definition, $t$ $1/2$-edges in $x$ are converted to 0-edges or 1-edges in $y$.
As the minimum vertex degree of $H(x)$ is at least $t + 2$, it follows that $V(y) = V(x)$. Furthermore, we have
\begin{equation}\label{red-child}
e_y(U, V(y)\setminus U) > |U|\mbox{ for all } U \subseteq V(y) \mbox{ of size at most }l_y.
\end{equation}
Therefore the minimum degree of $H(y)$ is $2$. There is no change in the 0-edges and 1-edges incident to vertices not in $V(x)$ when we move from $x$ to $y$. Therefore $y$ satisfies (\ref{one-x-inP}) and (\ref{two-x-inP}) and $y \in P$. We set $l_y = l_x$.

Note that $x^*$ satisfies the above assumption. Therefore the root is given red children, and we repeatedly apply the above rule till we have a node of $\mathcal{T}^A$ for which the above assumption does not hold. Assume again that the node is labeled by $x$.

Then we can assume that there exists $U \subseteq V(x)$ with size at most $ \ell_x$ such that $e_x(U, V(x) \setminus U) \leq (t + 1)|U|$.  Assume $U$ is a maximal set with this property. Then,
\begin{equation}\label{blue-child}
e_x(W, V(x)\setminus (U \cup W)) > (t+1)|W| \mbox{ for all  } W \subseteq V(x) \setminus U \mbox{ with } |W| \leq \ell_x - |U|.
\end{equation}
If (\ref{blue-child}) were not true for some $W \neq \emptyset$, then $U\cup W$ would violate the maximality of $U$. 
Let $Y$ be the set of all $y$ satisfying the system of odd-parity equations for vertices in $U$ such that $y_{uv} \in \{0,1\}$ for all $uv \in E(x)$ incident to $U$ and $y_{uv} = x_{uv}$ elsewhere. These points are the blue children we give $x$. We will now show that $Y$ is nonempty and the points in $Y$ average to $x$. Consider any $uv \in E(x)$ such that $u \in U$. If we delete edge $uv$ from $H(x)$ to obtain the subgraph $H'$ (with vertex set $V(x)$), then $e_{H'}(W, V(H') \setminus W) > 0$ for all subsets $W \subseteq U$. 
This last inequality follows from equation (\ref{red-child}) with $y$ replaced by $x$ in the equation, since $x$ is a red child of its parent.
Therefore, if we consider the odd-parity equations for nodes in $U$ with all variables outside $H(x)$ fixed to 0 or 1, then Lemma \ref{lem:exist_sol} implies that there is a solution $w$ to these equations with $w_{uv} = 0$, and another solution $w'$ with $w'_{uv} = 1$. By appending the values $x_{uv}$ for edges not incident with $U$, we get a 
$y \in Y$ with $y_{uv} = 0$ and a $y' \in Y$ with $y'_{uv} = 1$. 
Then, Lemma \ref{lem:avg_of_sols} ensures that these solutions average to $x$. 

Set $\ell_y = \ell_x - |U|$ for every child $y \in Y$. If $|U| < \ell_x$, then
$l_y > 0$, and we conclude from (\ref{blue-child}) that $V(y) = V(x) \setminus U$ and the assumptions for red children are satisfied by $y$.
Therefore every vertex $v \in V(x) \setminus U$ has degree at least $2$ in $H(y)$. This implies every inequality of $P$ is satisfied, since the equations corresponding to $U$ are satisfied by the above construction, and all other equations corresponding to nodes in $V(x) \setminus U$ have at least two $\frac{1}{2}$s, so the corresponding inequalities in the description of $P$ are satisfied. Finally, all vertices in $V \setminus V(x)$ have none of their incident edge values changed.  Therefore $y \in P$.
If $|U| = \ell_x$, then $\ell_y = 0$ and we call $y$ a leaf of the tree. Thus all leaves are blue children.

We next show that there exists a path in $\mathcal{T}^A$ from the root to a leaf node that passes through at least $\frac{c - (t + 1)}{2t}|V|$ red children. Consider any path from the root to a leaf node in $\mathcal{T}^A$, and let $$S = \bigcup_{U \text{ used for blue children in path }} U.$$ Since $S$ has size $|V| / 2$, we know that $e(S, V \setminus S) > c|S|$ by the edge expansion property of $G$. Note that all of the edges in this cut are set to 0-1 in the leaf node, and that at most $(t + 1)|S|$ of them were set to 0-1 when traversing to blue children. Then, we know that at least $(c - (t + 1))|S| = (c - (t + 1))\frac{|V|}{2}$ of them were set to 0-1 when traversing to red children, and each traversal to a red child sets at most $t$ edges to 0 or 1.

Finally, note that the parent nodes of the leaves of $\mathcal{T}^A$ are labeled by points that are contained in $P$, and also that the leaves are all blue children. Therefore, removing all the leaves of $\mathcal{T}^A$ results in a tree $\mathcal{T}$ that satisfies the conditions of a certificate tree, but also has the same number of red children in any path from the root to a leaf as in $\mathcal{T}^A$. The result follows. 
\end{proof}


\bibliography{bib}
\bibliographystyle{alpha}

\end{document}